\documentclass[reqno]{amsart}

\usepackage[T1]{fontenc}
\usepackage{lmodern}

\usepackage{amssymb, amsmath, amsthm}
\usepackage[mathscr]{eucal}
\usepackage{upgreek}
\usepackage{colonequals}

\usepackage{tikz}
\usetikzlibrary{arrows}
\tikzset{>=stealth}

\usepackage{parskip}

\usepackage[colorlinks=true, linkcolor=blue, citecolor=magenta, linktoc=page]{hyperref}

\newcommand{\N}{\mathbb{N}}
\newcommand{\Z}{\mathbb{Z}}
\newcommand{\C}{\mathbb{C}}
\newcommand{\R}{\mathbb{R}}
\newcommand{\KK}{\mathbb{K}}

\newcommand{\m}{\mathfrak{m}}

\newcommand{\K}{\mathrm{K}}
\newcommand{\HH}{\mathrm{H}}
\newcommand{\D}{\mathsf{D}}
\newcommand{\Db}{\mathsf{D}^b\kern-.5pt}
\newcommand{\Kb}{\mathsf{K}^b\kern-.5pt}
\newcommand{\Cb}{\mathsf{C}^b\kern-.5pt}

\newcommand{\Dfl}{\mathsf{D}^b_{\mathrm{fl}}\kern-.5pt}

\newcommand{\cA}{\mathcal{A}}

\newcommand{\cE}{\mathcal{E}}
\newcommand{\cF}{\mathcal{F}}
\newcommand{\cH}{\mathcal{H}}

\newcommand{\cS}{\mathcal{S}}
\newcommand{\cT}{\mathcal{T}}

\newcommand{\fE}{\mathscr{E}}

\newcommand{\Hom}{\mathrm{Hom}}
\newcommand{\Ext}{\mathrm{Ext}}
\newcommand{\End}{\mathrm{End}}

\newcommand{\RHom}{\mathbf R\Hom}

\newcommand{\op}{\mathrm{op}}

\newcommand{\UL}{\Uplambda}
\newcommand{\km}{\kappa}

\newcommand{\<}{\langle}
\renewcommand{\>}{\rangle}

\newcommand{\length}[1]{\operatorname{length}_{\kern-1pt#1}}

\newcommand{\SFan}{\operatorname{\mathsf{SFan}}}
\newcommand{\TCone}{\mathsf{TCone}}

\let\mod\relax
\DeclareMathOperator{\mod}{\mathsf{mod}}
\DeclareMathOperator{\silt}{\mathsf{2-silt}}
\DeclareMathOperator{\psilt}{\mathsf{2-presilt}}
\DeclareMathOperator{\fl}{\mathsf{fl}}
\DeclareMathOperator{\proj}{\mathsf{proj}}
\DeclareMathOperator{\Spec}{Spec}
\DeclareMathOperator{\cone}{\mathsf{cone}}
\DeclareMathOperator{\MM}{\mathsf{MM}}
\DeclareMathOperator{\Cl}{Cl}

\newtheorem{theorem}{Theorem}[section]
\newtheorem{lemma}[theorem]{Lemma}
\newtheorem{prop}[theorem]{Proposition}

\newtheorem{introthm}{Theorem}
\newtheorem{introprop}[introthm]{Proposition}

\theoremstyle{definition}

\theoremstyle{remark}
\newtheorem{rem}[theorem]{Remark}
\newtheorem{example}[theorem]{Example}

\title[Stability over cDV singularities and other complete local rings]{Stability over cDV singularities and other complete local rings}

\author[O. van Garderen]{Okke van Garderen}

\address{Max-Planck Institute for Mathematics, Vivatsgasse 7 Bonn}

\email{ogiervangarderen@gmail.com}



\begin{document}

\begin{abstract}
  We characterise subcategories of semistable modules for noncommutative minimal models of compound Du Val singularities, including the non-isolated case. We find that the stability is controlled by an infinite polyhedral fan that stems from tilting theory, and which can be computed from the Dynkin diagram combinatorics of the minimal models found in the work of Iyama--Wemyss. In the isolated case, we moreover find an explicit description of the deformation theory of the stable modules in terms of factors of the endomorphism algebras of 2-term tilting complexes. To obtain these results we generalise a correspondence between 2-term silting theory and stability, which is known to hold for finite dimensional algebras, to the much broader setting of algebras over a complete local Noetherian base ring.
\end{abstract}

\maketitle

\section{Introduction}

Stability conditions are an important tool in algebra, as they provide an avenue for studying wild representation theory through moduli spaces of (semi)stable objects. In some cases one can extract noncommutative Donaldson--Thomas invariants \cite{Szen} from these moduli spaces that give a ``virtual count'' of the semistable objects, which takes the deformation theory of these objects into account.

The aim of this paper is to understand stability conditions from this perspective for algebras with geometric significance: noncommutative minimal models for compound Du Val (cDV) singularities. These are a three-dimensional analogue of the noncommutative resolutions of ADE surface singularities encountered in the McKay correspondence, which are part of a noncommutative approach to the minimal model program for threefolds \cite{HomMMP}. We work in a wide setting, which includes both Van den Bergh's NCCRs \cite{VdBNCCR}, as well as the more general noncommutative minimal models studied by Iyama--Wemyss \cite{IyWeAusRei,IyWeMemoir}.

The (semi)stable modules for these noncommutative minimal models have only been classified in a limited number of cases, due to the complexity of the representation theory involved. Most notable is the work of Nagao--Nakajima \cite{NN}, who show that the stability of the ``conifold'' is controlled by a hyperplane arrangement. As a consequence, the enumerative theory of these examples is now well-established \cite{Szen,DaMeFlops}. In our recent work \cite{vG} we were able to uncover the enumerative theory for a new family of examples, by showing that the stability is again controlled by a hyperplane arrangement, coming from the 2-term tilting theory of the noncommutative minimal model. Moreover, we showed that these tilting complexes determine the deformation theory of the stable modules.

In this paper we further develop this connection between tilting and stability and thereby give a complete picture of the stability over \emph{any} cDV singularity, via the characterisation of the 2-term silting theory in the work of Iyama-Wemyss \cite{IyWeMemoir}. In fact, by applying recent results of Kimura \cite{Kimura} we are able to work in a much more general setting of 2-term \emph{silting} theory for algebras over any complete local base ring. The paper is therefore split into two parts: the first containing general results about silting and stability for algebras over a complete local ring, and the second on the application of these results to cDV singularities.

\subsection{Stability and silting over a complete local ring}

We work over a complete local Noetherian commutative ring $(R,\m)$, and consider module finite $R$-algebras $\UL$ with $R$ central. Our goal is to characterise semistable modules of such algebras for King stability conditions parametrised by K-theory vectors
\[
  \uptheta \in \K_0(\proj \UL)_\R \colonequals \K_0(\proj \UL)\otimes \R.
\]
To do this we leverage a certain duality between King stability and silting theory, building on the work of various authors \cite{Asai,BST1,Yur} in the setting of finite dimensional algebras. Concretely, we construct a \emph{silting fan} 
\[
  \SFan{\UL} = \bigcup_{T\in\silt\UL}\cone{T}
\]
as the union of closed cones in $\K_0(\proj \UL)_\R$ generated by basic 2-term silting complexes, and show that for each $\uptheta$ in this fan the category $\cS_\uptheta \subset \fl \UL$ of finite length $\uptheta$-semistable modules can be characterised via an orthogonality condition defined by a certain silting complex and its direct summands.

In Proposition \ref{prop:chamberwalls} we show that the silting fan is indeed a fan, whose faces are the \emph{strict} cones of basic 2-term \emph{pre}silting complexes.
\begin{introprop}[{Proposition \ref{prop:chamberwalls}}]
  Let $\UL$ be a module finite algebra over a complete local Noetherian ring. Then $\SFan{\UL}$ is a polyhedral fan with decomposition
  \[
    \SFan{\UL} = \bigsqcup_{T'\in \psilt\UL} \cone^\circ\kern-1pt{T'},
  \]
  where each face $\cone^\circ\kern-1pt{T'}$ is the \emph{strict} cone of a 2-term presilting complex.
\end{introprop}

The above is well-known in the finite dimensional setting, where it follows from the work of Demonet--Iyama--Jasso \cite{DIJ}. To lift this to the complete local setting, we apply a reduction theorem due to Kimura \cite{Kimura}, which allows us to compare the silting theories of the algebra $\UL$ and its finite dimensional fibre $\overline \UL \colonequals \UL\otimes_R R/\m$ over the closed point in $\Spec R$. In particular, we find that the structure of $\SFan{\UL}$ depends only on the fibre $\overline\UL$, which yields the following.

\begin{introprop}[{Proposition \ref{prop:fanquot}}]
  If $I\subset \m$ is an ideal, the base-change over the quotient $R\to R/I$ identifies the silting fans $\SFan{\UL}$ and $\SFan(\UL/I\UL)$.
\end{introprop}

For finite dimensional algebras such as $\overline \UL$, the semistable objects for stability conditions in $\SFan{\overline \UL}$ can be described via certain torsion pairs, as shown by Br\"ustle--Smith--Treffinger \cite{BST1} in the language of $\tau$-tilting theory, and by Yurikusa \cite{Yur} and Asai \cite{Asai} in terms of silting theory. One interpretation of these torsion pairs is as an orthogonality condition in the derived category between semistable modules and 2-term (pre)silting complexes. Our main technical result is that an analogous orthogonality condition on the derived category $\D(\UL)$ also controls the stability on the subcategory $\fl\UL$ of finite length $\UL$-modules.

\begin{introthm}[Theorem \ref{thm:identS}]
  Let $T\in \silt\UL$ be the Bongartz completion of a summand $T' \subset T$. For all $\uptheta\in \cone^\circ\kern-1pt{T'}$ the $\uptheta$-semistables form the subcategory
  \[
    \cS_\uptheta = \{M\in \fl\UL \mid \Hom_{\D(\UL)}(T,M[1]) = 0,\ \Hom_{\D(\UL)}(T',M) = 0\}.
  \]
  Moreover, $\cS_\uptheta$ is an abelian length category containing $|T|-|T'|$ simple objects, which are precisely the $\uptheta$-stable modules.
\end{introthm}

The orthogonality condition gives a description of $\cS_\uptheta$ as a subcategory inside the heart of a t-structure on the derived category $\Db(\fl \UL)$, which is induced by $T$. Via a derived Morita equivalence, this heart can be identified with the module category $\fl \End_{\Kb(\UL)}(T)$, leading to the following description of $\cS_\uptheta$.

\begin{introthm}[Theorem \ref{thm:stabmod}]
  Let $T$ be the Bongartz completion of a summand $T'\subset T$. Then for all $\uptheta\in \cone^\circ\kern-1pt{T'}$
  there is an equivalence of abelian categories
  \[
    \cS_\uptheta \xrightarrow{\ \sim\ } \fl \End_{\Kb(\UL)}(T)/(e),
  \]
  where $e\colon T\to T'\to T$ is the idempotent projecting onto the summand $T'\subset T$.
\end{introthm}

One way to interpret this theorem is through the lens of noncommutative deformation theory: for each $\uptheta \in \cone^\circ\kern-1pt{T'}$ the theorem implies that $\End_{\Kb(\UL)}(T)/(e)$ pro-represents Laudal's \cite{Laudal} deformation functor for the set of $\uptheta$-stable modules.

\subsection{Results about cDV singularities}
We now discuss the applications of the above results in the setting of cDV singularities. 
In \cite{IyWeAusRei} Iyama--Wemyss introduced a type of noncommutative minimal model for a cDV singularity $R$, in the form of endomorphism algebras
\[
  \Lambda = \Lambda_M \colonequals \End_R(M)
\]
of a certain distinguished set of reflexive modules, which are related by mutation.
In their follow up work \cite{IyWeMemoir} they moreover characterise the tilting theory of these minimal models by relating them to the 2-dimensional case: if $g \in R$ is a sufficiently generic section then any slice
\[
  \Lambda \to \overline \Lambda = \Lambda/g\Lambda
\]
yields a noncommutative \emph{partial} resolution of a simple surface singularity as in the McKay correspondence, and they show that certain tilting modules of the form $\Hom_R(M,N)$ descend to unique 2-term tilting complexes for $\overline \Lambda$.
The latter can be described via Dynkin diagram combinatorics: each 2-term tilting complex of $\overline \Lambda$ corresponds to a unique chamber in an \emph{intersection hyperplane arrangement}
\[
  X_{\Upgamma,J} \subset \R^J,
\]
which is determined by a Dynkin type $(\Upgamma,J)$ associated to the slice. 
If $R$ is an \emph{isolated} singularity, they then show that this establishes a bijection between chambers in $X_{\Upgamma,J}$ and 2-term tilting complexes for $\Lambda$.
Using Proposition \ref{prop:fanquot}, we are able to extend this result to the non-isolated setting.

\begin{introprop}[Proposition \ref{prop:cDVfan}]
  Let $\Lambda = \Lambda_M$ be a noncommutative minimal model of a cDV singularity with Dynkin type $(\Upgamma,J)$. 
  Then $\SFan{\UL}$ is identified with the intersection arrangement $X_{\Upgamma,J}$ via an isomorphism $\K_0(\proj \Lambda)\simeq \R^J$.
\end{introprop}

In this setting all silting complexes are tilting by \cite[Appendix A]{KimMiz}, and therefore the above proposition yields a bijection between 2-term tilting complexes and chambers in $X_{\Upgamma,J}$.

With the silting fan identified, Theorem \ref{thm:identS} directly yields a classification of all K-theory vectors for which there exists semistable modules. 
This again holds both for isolated as well as non-isolated singularities.

\begin{introprop}[Proposition \ref{prop:cDVnumstables}]
Let $\Lambda = \Lambda_M$ be a noncommutative minimal model of a cDV singularity with Dynkin type $(\Upgamma,J)$. For $\uptheta$ contained in a codimension $k$ face of $X_{\Upgamma,J}$ there are precisely $k$ stable modules in $\cS_\uptheta \subset \fl\Lambda$. In particular:
  \begin{itemize}
  \item if $\uptheta$ lies in a chamber, then $\cS_\uptheta = 0$,
  \item if $\uptheta$ lies generically on a hyperplane, then there is a unique $\uptheta$-stable module.
  \end{itemize}
\end{introprop}

This result gives a strong grip on the stability of finite length modules, because the intersection arrangements, though infinite, can be described concretely using Dynkin diagram combinatorics. As an example, one can consider the arrangement associated to a choice of two nodes in the extended $D_4$ Dynkin diagram:
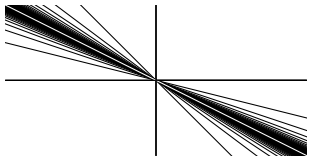
\[
    \begin{tikzpicture}[baseline=(current bounding box.center)]
      \begin{scope}
        \clip (-2,-1) rectangle (2,1);
        \draw (0,0) to (8,0);
        \foreach \i in {0,1,2,3,4,5}
        {
          \draw   (16*\i,   -8*\i-4) to (-16*\i , 8*\i+4);  
          \draw   (16*\i,   -8*\i+4) to (-16*\i , 8*\i-4);  
          \draw   (12*\i-6, -6*\i) to (-12*\i+6 ,6*\i);  
          \draw   (12*\i+6, -6*\i) to (-12*\i-6, 6*\i);  
        }
        \foreach \i in {6,7,8,9,10,11,12,13,14,15,16,17,18,19,20,21,22,23,24,25,26,27,28,29,30,35,40,45,50,55,60,70,80,90}
        {
          \draw   (4*\i,   -2*\i-1) to (-4*\i , 2*\i+1);  
          \draw   (4*\i,   -2*\i+1) to (-4*\i , 2*\i-1);  
          \draw   (2*\i-1, -\i) to (-2*\i+1, \i);  
          \draw   (2*\i+1, -\i) to (-2*\i-1, \i);  
        }
        \draw[color=white] (12, -6) to (-12,6);  
        \draw (-2,0) to (2,0);
        \draw (0,-1) to (0,1);
      \end{scope}
    \end{tikzpicture}
\]
The proposition implies that for a $\Lambda$ with this Dynkin-type, the subcategories $\cS_\uptheta$ behave constructibly with respect to the above wall-and-chamber structure: $\cS_\uptheta$ is trivial for $\uptheta$ in a chamber, $\cS_\uptheta = \fl\Lambda$ if $\uptheta = 0$, and if $\uptheta$ lies in the complement of $0$ inside a wall then $\cS_\uptheta$ consists of self-extensions of a unique $\uptheta$-stable module.

The constructible behaviour of $\cS_\uptheta$ also holds for higher dimensional arrangements, such as those illustrated in \cite{IyWeMemoir}, and implies that the K-theory classes of semistable modules in $\K_0(\fl \Lambda)$ are contained in a lattice that is \emph{dual} to the intersection arrangement. This has implications for enumerative theories, because it determines which classes in $\K_0(\fl\Lambda)$ will contribute to the enumerative invariants of $\Lambda$.

In the isolated setting, the correspondence of \cite{IyWeMemoir} yields an explicit tilting complex for each chamber in $X_{\Upgamma,J}$, and the endomorphism algebras are again noncommutative minimal models. Therefore Theorem \ref{thm:stabmod} can be used in this setting to give an explicit description of the subcategories of semistable modules.

\begin{introthm}[Theorem \ref{thm:cDVstabs}]
  Let $\Lambda = \Lambda_M$ be a noncommutative minimal model of an \emph{isolated} cDV singularity. Then for all $\uptheta \in X_{\Upgamma,J}$ there is an equivalence of abelian categories
  \[
    \cS_\uptheta \simeq \fl \Lambda_N/(e),
  \]
  for some noncommutative minimal model $\Lambda_N$ and idempotent $e\in \Lambda_N$.
\end{introthm}

Hence, the deformation theory of $\uptheta$-stable modules of a fixed minimal model $\Lambda_M$ is pro-represented by the factor algebras $\Lambda_N/(e)$ of the other noncommutative minimal models, at least for $\uptheta$ contained in the fan. This has strong implications for the enumerative geometry of flops, which we will cover in a separate paper.

\subsection{Acknowledgements}

While this work was carried out the author was a PhD student at the University of Glasgow, who the author thanks for their financial support. The author would also like to thank Michael Wemyss for helpful discussions.

\section{Stability and silting}

In this section we fix a complete local Noetherian commutative ring $(R,\m)$ with residue field $\km=R/\m$, and an $R$-algebra $\UL$ such that $R$ is central in $\UL$ and $\UL$ is finite as an $R$-module. Equivalently, $\UL$ is a coherent sheaf of algebras over $\Spec R$, whose fibre over the unique closed point $\Spec \km$ is the finite dimensional algebra
\[
  \overline\UL \colonequals \UL/\m\UL \simeq \UL\otimes_R \km.
\]
In what follows we compare the stability conditions and silting theory for $\UL$ and $\overline \UL$ via functors induced by the quotient map $\UL \to \overline \UL$.

\subsection{King stability over a complete local base} King \cite{King} defined stability conditions on any abelian category $\cA$ in terms of \emph{linear forms} $\K_0(\cA)\to \R$ on its Grothendieck group.  Here we want to define such stability conditions on the abelian categories $\cA = \fl \UL$ and $\cA = \fl \overline \UL$ of finite length modules over $\UL$ and $\overline \UL$, in a compatible way. To do this we use a pairing of $\K_0(\fl \UL)$ with the real K-theory space
\[
  \K_0(\proj \UL)_\R \colonequals \K_0(\proj \UL) \otimes_\Z \R,
\]
where $\proj \UL$ denotes the category of finitely generated projective modules, and similarly for $\overline \UL$. Before defining this pairing we need the following elementary lemma, where we use the notation $\length{R} M$ for the length of $M$ as an $R$-module.
\begin{lemma}\label{lem:homfl}
  Let $P\in \proj \UL$ and $M\in \fl\UL$, then $\Hom_\UL(P,M)$ is a finite length $R$-module.
\end{lemma}
\begin{proof}
  Because $P$ is a finitely generated projective, it is a direct summand of $\UL^{\oplus n}$ for some $n$. Because $\Hom_\UL(P,M)$ is a direct summand of $\Hom_\UL(\UL^{\oplus n},M)$, the additivity of the length therefore implies
  \[
    \length{R} \Hom_\UL(P,M) \leq \length{R} \Hom_\UL(\UL^{\oplus n},M) = n \cdot \length{R} M.
  \]
  Because $\UL$ is finite as an $R$-module, a finitely generated $\UL$-module has finite length over $\UL$ if and only if it has finite length over the base ring $R$. In particular, $M\in \fl \UL$ implies $\length{R} M <\infty$, and the result follows.
\end{proof}
Lemma \ref{lem:homfl} implies that there is a well-defined mapping $\proj \UL \times \fl \UL \to\N$ which sends a pair $(P,M)\in \proj \UL \times \fl \UL$ to $\length{R} \Hom_\UL(P,M)$. Because the length is additive over short exact sequences, it induces the following pairing on K-theory.
\begin{lemma}
  The assignment $(P,M) \mapsto \length{R}\Hom_\UL(P,M)$ extends to a pairing 
  \[
    \<-,-\>_\UL \colon \K_0(\proj\UL)_\R \otimes \K_0(\fl\UL) \to \R.
  \]
\end{lemma}
\begin{proof}
  The Grothendieck groups $\K_0(\proj\UL)_\R$ and $\K_0(\fl \UL)$ are generated by the symbols $[P]$ for $P\in \proj\UL$, and $[M]$ for $M\in\fl \UL$ respectively. Therefore, we can define $\<-,-\>_\UL$ by extending bi-linearly:
  \[
    \textstyle\<\sum_i \uptheta_i[P_i],\sum_j n_j[M_j]\> = \sum_{ij} \uptheta_i n_j \length{R}\Hom_\UL(P_i,M_j),
  \]
  and it suffices to check that this form $\<-,-\>_\UL$ respects the relations in the Gro\-then\-dieck group. For $P\in \proj \UL$, the functor $\Hom_\UL(P,-)$ is exact and therefore maps every short exact sequence $0\to N\to M \to M/N \to 0$ in $\fl \UL$ to a short exact sequence in $\fl R$, and hence by the additivity of the length:
  \[
    \length{R} \Hom_\UL(P,M) = \length{R} \Hom_\UL(P,N) + \length{R} \Hom_\UL(P,M/N).
  \]
  It follows that $\<[P],[N]-[M]+[M/N]\>_\UL = 0$. Likewise, every admissible exact sequence $Q \hookrightarrow P \twoheadrightarrow P/Q$ in the exact category $\proj \UL$ induces an exact sequence for every $M\in \fl\UL$
  \[
    \begin{tikzpicture}
      \node (A) at (-6.5,0) {$0$};
      \node (B) at (-4.25,0) {$\Hom_\UL(P/Q,M)$};
      \node (C) at (-1,0) {$\Hom_\UL(P,M)$};
      \node (D) at (-4.25,-20pt) {$\Hom_\UL(Q,M)$};
      \node (E) at (-1,-20pt) {$\Ext^1_\UL(P/Q,M)$};
      \node (F) at (.5,-20pt) {$=0$};
      \draw[->] (A) to (B);
      \draw[->] (B) to (C);
      \draw[->] (C.east) to ([xshift=5pt]C.east) arc (90:-90:5pt) to ([yshift=10pt,xshift=-5pt]D.west) arc (90:270:5pt) to (D.west);
      \draw[->] (D) to (E);
    \end{tikzpicture}
  \]
  which implies that $\<[Q]-[P]+[P/Q],[M]\>_\UL = 0$.
\end{proof}
Given any vector $\uptheta\in \K_0(\proj \UL)_\R$, the pairing now yields a linear form 
\[
  \<\uptheta,-\>\colon\K_0(\fl\UL)\to \R,
\]
and therefore a King stability condition on $\fl \UL$. Following King's definition \cite[Definition 1.1]{King}, a module $M\in \fl \UL$ is $\uptheta$-\emph{semistable} if $\<\uptheta,[M]\> = 0$ and $\<\uptheta,[N]\> \leq 0$ for every submodule $N\subset M$, and is additionally called $\uptheta$-\emph{stable} if\footnote{Note that this definition guarantees that the zero module is semistable but not stable, just as it is semisimple but not simple.} $M\neq 0$ and the inequality is strict for all proper nonzero submodules. We write $\cS_\uptheta \subset \fl\UL$ for the full subcategory of semistable modules:
\[
  \cS_\uptheta = \left\{M\in \fl \UL \mid M \text{ is $\uptheta$-semistable}\right\}.
\]
In the finite dimensional setting discussed in \cite{BST1,Asai} these subcategories are known to be \emph{wide} subcategories, i.e. finite length subcategories which are closed under kernels, cokernels, and extensions. In our setting the analogous result holds.

\begin{lemma}\label{lem:wide}
  For all $\uptheta\in \K_0(\proj \UL)_\R$ the subcategory $\cS_\uptheta \subset \fl \UL$ is a wide subcategory, and its simple objects are exactly the $\uptheta$-stable modules.
\end{lemma}
\begin{proof}
  This follows directly from \cite[Proposition 2.20]{BST2}, after translating the King stability condition $\<\uptheta,-\>$ to the associated Rudakov stability function (as in \cite{Rudakov})
  \[
    [M] \mapsto \tfrac1\pi \operatorname{Arg}\left(-\<\uptheta,[M]\> + i \length{R} M \right).
  \]
  The subcategory $\cS_\uptheta$ is then identified with a subcategory of Rudakov-semistable objects of phase $\tfrac12$, as used in \cite[Proposition 2.20]{BST2}.
\end{proof}

One similarly obtains a pairing $\<-,-\>_{\overline\UL}$ for the finite dimensional fibre $\overline \UL$, and each vector $\overline\uptheta \in \K_0(\proj\overline\UL)$ induces a stability condition on $\fl\overline\UL$, with an associated wide subcategory $\cS_{\overline\uptheta} \subset \mod\overline\UL$. 
In order to relate stability conditions for $\UL$ and $\overline \UL$ we use the extension/restriction of scalars functors
\[  
  \overline {(-)}\colonequals  (-)\otimes_\UL\overline\UL \colon \proj \UL \to \proj \overline\UL, \quad (-)_\UL \colon \fl \overline \UL \to \fl \UL,
\]
which induce isomorphisms on K-theory, because the ideal $\m\UL \subset \UL$ is contained in the radical of $\UL$.
The functors are related by the adjunction $\Hom_\UL(-,(-)_\UL) \simeq \Hom_{\overline\UL}(\overline{(-)},-)$ and, as the following proposition shows, this adjunction yields a compatibility between the pairings that relates the stability conditions on $\fl \UL$ and $\fl \overline \UL$.
\begin{prop}\label{prop:stabred}
  Let $\uptheta\in \K_0(\proj \UL)_\R$ with image $\overline \uptheta \in \K_0(\proj\overline\UL)_\R$, then the functor $(-)_\UL$ defines an exact embedding $\cS_{\overline\uptheta} \longrightarrow \cS_\uptheta$ which identifies the $\overline\uptheta$-stable modules with the $\uptheta$-stable modules. In particular $\cS_\uptheta$ is the extension closure 
  \[
    \cS_\uptheta = \<(\cS_{\overline\uptheta})_\UL\>.
  \]
\end{prop}
\begin{proof}
  Write $\uptheta$ as a linear combination $\uptheta = \sum_i \uptheta_i [P_i]$ for $P_i\in\proj \UL$, so that its image is $\overline\uptheta = \uptheta_i[\overline P_i]$. For any $\overline M\in\fl\overline \UL$, the adjunction yields
  \[
    \begin{aligned}
      \<\uptheta, [\overline M_{\kern-1pt\UL}] \>_\UL 
      &= {\textstyle\sum_i} \uptheta_i \cdot \length{R}\Hom_\UL(P_i,\overline M_{\kern-1pt\UL}) 
      \\&= {\textstyle\sum_i} \uptheta_i\cdot  \length{R} \Hom_{\overline \UL}(\overline P_i,\overline M) 
      = \<\overline\uptheta,[\overline M]\>_{\overline \UL}.
    \end{aligned}
  \]
  Because the submodules of the restriction $\overline M_\UL$ are precisely the restrictions of submodules of $\overline M$, it then follows that that $\overline M$ is $\overline \uptheta$-(semi)stable if and only if its image $\overline M_{\kern-1pt \UL}$ is $\uptheta$-(semi)stable. Hence $(-)_\UL$ restricts to an embedding $\cS_{\overline\uptheta} \to \cS_\uptheta$.

  To see that every $\uptheta$-stable module is in the image, recall (see e.g. \cite[Theorem 1]{Rudakov}) that the endomorphism ring of a stable object in an abelian category is a division ring. Therefore, if $M\in \cS_\uptheta$ is a stable module, then for any $z\in\m$ the homomorphism $z\cdot \colon M\to M$ of multiplication by $z$ is either $0$ or an isomorphism. Because $\m$ maps into the Jacobson radical of $\UL$, Nakayama's lemma implies $z\cdot$ is the zero map for all $z\in \m$, which shows that $M \simeq \overline M_\UL$ is in the image of $(-)_\UL$.

By Lemma \ref{lem:wide} the category $\cS_\uptheta$ is the extension closure of the set of $\uptheta$-stable modules. As these are all in the image $(\cS_{\overline\uptheta})_\UL$ it then follows that $\cS_\uptheta = \<(\cS_{\overline\uptheta})_\UL\>$.
\end{proof}

\subsection{The silting fan}
Let $\Kb(\UL) \colonequals \Kb(\proj\UL)$ denote the homotopy category of complexes of projectives. Recall that $T\in \Kb(\UL)$ is \emph{presilting} if
\[
  \Hom_{\Kb\kern-.5pt(\UL)}(T,T[i]) = 0 \quad \forall i>0,
\]
and is \emph{silting} if it additionally generates $\Kb(\UL)$ as a triangulated category. 
Because $\UL$ is module finite over a complete local ring, $\Kb(\UL)$ is known to be Krull-Schmidt: every complex $T\in \Kb(\UL)$ splits as a direct sum $T = T_1\oplus \ldots\oplus T_k$ of indecomposables. We write $|T|$ for the number of indecomposables in the Krull-Schmidt decomposition. A complex is \emph{basic} if it has no repeated indecomposable summands, and is a \emph{2-term} complex if it is supported in degrees $-1$ and $0$. 
The sets of isomorphism classes of basic 2-term presilting/silting complexes will be denoted by
\[
  \psilt \UL,\quad \silt \UL
\]
respectively, and in what follows we will take the liberty to identify the elements of $\psilt\UL$ and $\silt\UL$ with a choice of representatives $T\in \Kb(\UL)$. 
By \cite[Theorem 2.11]{AiIy} the set $\silt \UL$ furthermore admits a partial order $\geq$, for which $T,T' \in \silt \Lambda$ satisfy $T \geq T'$ if and only if
\[
  \Hom_{\Kb(\Lambda)}(T,T'[i]) = 0\quad \forall i>0.
\]
(Pre)silting complexes for $\overline \UL$ are defined in an analogous way, and similarly yield a set $\psilt\overline \UL$ and a poset $\silt\overline\UL$.
The following result of Kimura shows that there is a relation between the 2-term silting theory of $\UL$ and $\overline \UL$.

\begin{prop}[{\cite[Proposition 4.2(b), Proposition 4.5(a)]{Kimura}}]\label{prop:kimura}
  Let $I \subset \m$ be an ideal, then the functor $-\otimes_\UL \UL/I\UL$ defines a surjection 
  \[
    \psilt \UL \to \psilt\UL/I\UL,
  \]
  which restricts to an equivalence of partially ordered sets
  \[
    \silt \UL \xrightarrow{\ \sim\ } \silt \UL/I\UL.
  \]
\end{prop}

Clearly, this result applies to the case $I=\m$, giving a bijection $\silt \UL \simeq \silt \overline \UL$. 
In what follows we will however also require such a bijection for the presilting complexes.
For this we use the \emph{Bongartz completion}, which is characterised as follows.

\begin{lemma}\label{lem:Bongartz}
  Let $T \in \psilt \UL$. Then there exists a complex $U \in\psilt\UL$, unique up to isomorphism, such that:
  \begin{itemize}
  \item $T\oplus U \in \silt \UL$,
  \item $T\oplus U \geq T\oplus U'$ for any other completion $T\oplus U' \in \silt \UL$
  \end{itemize}
  The complex $T\oplus U$ is called the Bongartz completion of $T$.
\end{lemma}
\begin{proof}
  This is \cite[Proposition 2.16]{Aihara} in the finite dimensional setting, which generalises to the module-finite setting by \cite[Lemma 4.2]{IJY} when the homotopy category is Krull-Schmidt.
  The complex $U$ can be constructed via an approximation sequence
  \[
    U \to T' \xrightarrow{f} \Lambda[1] \to U[1],
  \]
  where $T'$ is in the additive closure $\operatorname{\mathsf{add}} T$ of $T$ and $f$ is a right approximation.
  If $U'$ is any other complex such that $T\oplus U' \in \silt\UL$, then $\Hom_{\Kb(\UL)}(T',U') = 0$ for any $T' \in \operatorname{\mathsf{add}} T$.
  Applying $\Hom_{\Kb(\UL)}(-,U')$ to the approximation sequence yields the long exact sequence
  \[
    \ldots \to \Hom_\UL(T',U'[i]) \to \Hom_\UL(U,U'[i]) \to \Hom_\UL(\Lambda,U'[i]) \to \ldots,
  \]
  from which it follows that $\Hom_\UL(U,U'[i]) = 0$ for all $i>0$, hence $T\oplus U \geq T\oplus U'$.

  Because $\Kb(\UL)$ is Krull-Schmidt, one can ensure that $T\oplus U$ is a basic completion by removing repeated summands from $U$.
\end{proof}

\begin{lemma}\label{lem:psiltbij}
  The map $\psilt\UL \to \psilt \overline\UL$ is a bijection.
\end{lemma}
\begin{proof}
  Suppose $T,T' \in \psilt \UL$ are basic 2-term presilting complexes such that $\overline T\simeq \overline T'$, and let $T\oplus U$ be the Bongartz completion of $T$.
  Then $\overline T \oplus \overline U$ is again a 2-term silting complex, and so is $\overline T' \oplus \overline U \simeq \overline T \oplus \overline U$.
  By \cite[Proposition 4.2(a)]{Kimura}, any 2-term complex $V\in \Kb(\UL)$ satisfies
  \[
    \Hom_{\Kb(\overline\UL)}(\overline V, \overline V[1]) = 0 \quad\Longleftrightarrow\quad
    \Hom_{\Kb(\UL)}(V, V[1]) = 0.
  \]
  so that $V$ is silting if and only if $\overline V$ is silting. In particular, this shows that the 2-term complex $V = T' \oplus U$ is silting.
  Because $\silt \UL \to \silt \overline\UL$ is an equivalence, it follows that the isomorphism $\overline T \oplus \overline U \simeq \overline T' \oplus \overline U$ lifts to an isomorphism $T\oplus U \simeq T'\oplus U$, and it then follows that $T\simeq T'$ because $\Kb( \UL)$ is Krull-Schmidt. Hence $\psilt\UL\to\psilt\overline\UL$ is injective, and therefore also bijective.
\end{proof}

\begin{rem}\label{rem:Bongartz}
  The above lemma implies that Kimura's reduction map is compatible with Bongartz completion: if $T\in\psilt\UL$ then any completion of $\overline T$ is of the form $\overline T\oplus\overline U$ for some $U\in\psilt\UL$, and this completion is maximal in $\silt\overline \UL$ precisely if $T\oplus U$ is maximal among the completions of $T$ in $\silt\UL$.
\end{rem}

Because $\overline \UL$ is a finite dimensional algebra, it follows by \cite{BST1,Asai} that the stability conditions on $\fl \overline\UL$ are characterised by a wall-and-chamber structure in $\K_0(\proj \overline \UL)_\R$ generated by the \emph{g-vector cones} of 2-term silting complexes.
To derive an analogous result for $\UL$, we will construct a similar wall-and-chamber structure in $\K_0(\proj \UL)_\R$ and compare it with the wall-and-chamber structure of $\overline \UL$ via the isomorphism on K-theory.

Given a presilting complex $T\in \psilt \UL$ with Krull-Schmidt decomposition of the form $T = T_1\oplus \ldots \oplus T_k$, let
\[
  \begin{aligned}
    \cone{T} \colonequals \left\{\textstyle \sum_{i=1}^k \uplambda_i[T_i] \mid \uplambda_i \geq 0\right\} \subset \K_0(\proj \UL)_\R,\\
    \cone^\circ\kern-1pt{T} \colonequals \left\{\textstyle \sum_{i=1}^k \uplambda_i[T_i] \mid \uplambda_i > 0\right\} \subset \K_0(\proj \UL)_\R,
  \end{aligned}
\]
denote respectively the cone and strict cone of $T$ in $\K_0(\proj\UL)_\R$. The cone of $T$ is polyhedral with faces given by the strict cones of its summands, i.e. $\cone{T}$ decomposes as a disjoint union
\[
  \cone{T} = \bigsqcup_{T'\subset T}\cone^\circ\kern-1pt{T'}.
\]
In particular, $\cone^\circ\kern-1pt{T}$ is the interior of $\cone{T}$ inside its linear span, and the boundary is a union of faces $\cone (T\setminus T_i)$ of the presilting complexes obtained by removing one of the indecomposable summands. 

Similarly, every presilting complex $\overline T\in \psilt \overline\UL$ defines a cone in $\K_0(\proj \overline \UL)_\R$, and we claim that this construction is compatible with the reduction map.

\begin{lemma}\label{lem:conereduct}
  For every $T\in \psilt \UL$, the strict cones $\cone^\circ\kern-1pt{T}$ and $\cone^\circ\kern-1pt{\overline T}$ are identified by $\overline{(-)}\colon \K_0(\proj\UL)_\R \xrightarrow{\sim} \K_0(\proj \overline \UL)_\R$.
\end{lemma}
\begin{proof}
  Let $T = T_1\oplus \ldots \oplus T_k$ be the Krull-Schmidt decomposition of some $T\in\psilt\UL$, then $\overline T$ has the Krull-Schmidt decomposition $\overline T = {\overline T}_{\kern-1pt 1}\oplus\ldots\oplus \overline T_{\kern-1pt k}$ as the $\overline T_{\kern-1pt i}$ are indecomposable by Proposition \ref{prop:kimura}. Hence
  \[
    \overline{\cone^\circ\kern-1pt{T}} = \textstyle\left\{\sum_{i=1}^k \uplambda_i [\overline T_{\kern-1pt i}] \mid \uplambda_i > 0 \right\} = \cone^\circ\kern-1pt{\overline T},
  \]
  as claimed. 
\end{proof}

It follows by \cite[Theorem 2.27]{AiIy} that for any $T \in \silt \UL$, the classes $[T_i]$ of the indecomposable summands form a basis for $\K_0(\proj \UL)_\R$, and $\cone^\circ\kern-1pt{T}$ is therefore an open subspace of $\K_0(\proj\UL)_\R$. In the finite dimensional setting it is known that these open subspaces form the chambers of a wall-and-chamber structure, with walls given by the cones of a common summand. Here we find a similar result.

\begin{prop}\label{prop:chamberwalls}
  For $T,U\in \psilt\UL$ their cones intersect in
  \[
    \cone{T} \cap \cone{U} = \cone{V}
  \]
  for some $V \in \psilt \UL$ which is a summand of both $T$ and $U$. In particular, if $T,U\in \silt\UL$ are distinct silting complexes, then
  \[
    \cone^\circ\kern-1pt{T}\cap \cone^\circ\kern-1pt{U} = \emptyset.
  \]
\end{prop}
\begin{proof}
  This is well-known in the finite dimensional setting: for 2-term tilting complexes it is a result by Hille \cite{Hille}, for 2-term silting complexes by Demonet--Iyama--Jasso \cite[Corollary 6.7]{DIJ}, and analogously for 2-term presilting complexes as in \cite[Proposition 3.9]{Asai}. Because the fibre $\overline \UL$ is finite dimensional, the reductions $\overline T,\overline U\in \psilt\overline\UL$ of a pair of presilting complexes $T,U\in\psilt\UL$ satisfy
  \[
    \cone{\overline T} \cap \cone{\overline U} = \cone{\overline V},
  \]
  for some $\overline V\in\psilt\overline\UL$ which is then a summand of $\overline T$ and $\overline U$. By Lemma \ref{lem:psiltbij} the object $\overline V$ lifts uniquely to some $V \in \psilt\UL$ which is a summand of $T$ and $U$. By Lemma \ref{lem:conereduct} it then also follows that 
  \[
    \cone{T} \cap \cone{U} = \cone{V}.
  \]
  
  If $T$ and $U$ are silting then they have the same number $|T| = |U|$ of indecomposable summands, so if $T\not\simeq U$ then any shared summand $V$ has $|V|<|T|$ and $|V|<|U|$. It follows that $\cone{V}$ is disjoint from $\cone^\circ\kern-1pt{T}$ and $\cone^\circ\kern-1pt{U}$.
\end{proof}

The proposition shows that the union 
\[
  \SFan{\UL} \colonequals \bigcup_{T\in \silt\UL} \cone{T}
\]
of the cones of 2-term silting complexes forms a polyhedral fan inside $\K_0(\proj \UL)$, which we will refer to as the \emph{silting fan} of $\UL$. The fan decomposes into a disjoint union
\[
  \SFan{\UL} = \bigsqcup_{T \in \psilt \UL} \cone^\circ\kern-1pt{T},
\]
of its faces, which are precisely the strict cones of 2-term \emph{pre}silting complexes. In the following section it will be useful to stratify the silting fan by codimension. If $d=\dim_\R \K_0(\proj \UL)_\R$ denotes the rank of the K-theory, we define for each $k=0,1,\ldots,d$ the codimension $k$ stratum as the disjoint union
\[
  \SFan_k{\UL} \colonequals \bigsqcup_{|T| = d-k} \cone^\circ\kern-1pt{T},
\]
over all $T\in \psilt\UL$ with $d-k$ indecomposable summands. As a consequence of Kimura's theorem, we find that the construction of the silting fan is compatible with any central reduction. More precisely, we have the following.

\begin{prop}\label{prop:fanquot}
  For every ideal $I\subset \m$, the silting fans $\SFan{\UL}$ and $\SFan(\UL/I\UL)$ are identified by $\K_0(\proj \UL)_\R \xrightarrow{\sim} \K_0(\proj \UL/I\UL)_\R$.
\end{prop}
\begin{proof}
  By Lemma \ref{lem:conereduct} the faces $\cone^\circ\kern-1pt{T}$ of the silting fan $\SFan{\UL}$ are mapped to the faces  $\cone^\circ\kern-1pt{\overline T}$ of the silting fan $\SFan{\overline \UL}$ of the fibre, and the faces are moreover in bijection by Lemma \ref{lem:psiltbij}. Hence, $\SFan{\UL}$ and $\SFan{\overline\UL}$ are identified as polyhedral fans. Likewise, if $I\subset \m$ then the reduction $\UL/I\UL \to \overline \UL$ identifies $\SFan(\UL/I\UL)$ with $\SFan{\overline \UL}$. Because the reduction of $\UL$ factors as $\UL\to\UL/I\UL \to \overline\UL$, it follows that the natural map $\K_0(\proj \UL)_\R \xrightarrow{\sim} \K_0(\proj\UL/I\UL)$ identifies $\SFan{\UL}$ with $\SFan(\UL/I\UL)$, as both fans map to $\SFan{\overline\UL}$ in $\K_0(\proj \overline\UL)_\R$.
\end{proof}

\subsection{Stability conditions inside the silting fan}

In the finite dimensional setting, a connection between silting theory and stability conditions has been established in the work of several authors \cite{BST1,Yur,Asai}: for a finite dimensional algebra $A$ and a stability condition in $\uptheta \in \SFan{A}$ it is shown that
\begin{itemize}
\item there are a finite number of $\uptheta$-stable modules,
\item the subcategory $\cS_\uptheta$ depends only on the face $\cone^\circ\kern-1pt{T}$ of the fan in which $\uptheta$ lies,
\item $\cS_\uptheta$ is equivalent as an abelian category to a module category for an algebra determined by $T \in \psilt A$ and its Bongartz completion.
\end{itemize}
In this section we show how to derive similar result for the algebra $\UL$ by reducing to the finite dimensional setting along the quotient $\UL \to \overline \UL$. Our strategy follows the approach of Asai \cite{Asai}: we recover the subcategories $\cS_\uptheta \subset \fl \UL$ by identifying the $\uptheta$-stable objects with simples in the heart of a t-structure associated to a silting complex. This t-structure can be identified as follows.

Let $T\in \silt\UL$, then $T$ can be viewed as a chain complex in $\Cb(\proj\UL)$, and therefore has a well-defined endomorphism DG algebra $\fE_T \colonequals (\End_\UL^\bullet(T),\partial)$. Because $T$ is silting, it is by definition also a generator for the derived category $\D(\UL)$ and therefore (see e.g. \cite[Theorem 14.2.29]{YekDer}) induces a \emph{derived Morita equivalence} 
\[
  \RHom_\UL(T,-)\colon \D(\UL) \longrightarrow \D(\fE_T)
\]
between $\D(\UL)$ and the derived category of DG modules over $\fE_T$. Moreover, it follows from Lemma \ref{lem:homfl} that this restricts to an equivalence on finite length objects
\[
  \RHom_\UL(T,-)\colon  \Db(\fl \UL) \longrightarrow \Dfl(\fE_T).
\]
Because $T$ is silting, the cohomology $\HH^i\fE_T = \Hom_{\Kb(\UL)}(T,T[i])$ of $\fE_T$ vanishes in positive degrees, and the derived category $\Dfl(\fE_T)$ has a standard t-structure with heart given by the complexes concentrated in degree $0$:
\[
  \HH^0(\Dfl(\fE_T)) \simeq \fl \HH^0\fE_T \simeq \fl \End_{\Kb( \UL)}(T).
\]
Pulling this t-structure back along the derived Morita equivalence, one obtains a t-structure on $\Db(\fl \UL)$ with heart
\[
  \cH_T = \{E\in \Db(\fl\UL) \mid \RHom_\UL(T,E) \in \fl\End_{\Kb( \UL)}(T) \subset \Dfl(\fE_T)\}.
\]
This heart $\cH_T  \subset \Db(\fl\UL)$ is a tilt of the standard heart $\fl \UL \subset \Db(\fl\UL)$ at a torsion pair, as defined by Happel-Reiten-Smal\o~ \cite{HRS}.

\begin{lemma}\label{lem:HRS}
  Let $T\in\silt \UL$, then the heart $\cH_T$ is the HRS-tilt
  \[
    \cT * \cF[1] \colonequals \{E\in \Db(\fl\UL) \mid \HH^0E \in \cT,\ \HH^{-1}\kern-1pt E \in \cF,\ \HH^i\kern-1pt E = 0\ \text{\normalfont otherwise}\}
  \]
  associated to the torsion pair 
  \[
    \begin{aligned}
      \cT &= \{M\in \fl \UL\mid \Hom_{\D(\UL)}(T,M[1]) = 0\},\\
      \cF &= \{M\in \fl \UL\mid \Hom_{\D(\UL)}(T,M) = 0\}.
    \end{aligned}
  \]
\end{lemma}
\begin{proof}
  By construction $\cH_T$ is the heart of the t-structure $(T[<0]^\bot,T[>0]^\bot)$, where 
  \[
    \begin{aligned}
      T[<0]^\bot \colonequals \{E\in\Db(\fl\UL)\mid \Hom_{\D(\UL)}(T,E[i]) = 0\ \forall i > 0\},\\
      T[>0]^\bot \colonequals \{E\in\Db(\fl\UL)\mid \Hom_{\D(\UL)}(T,E[i]) = 0\ \forall i < 0\}.
    \end{aligned}
  \]
  It was shown in \cite{HKM} in a more general setting that this t-structure is induced by the torsion pair $(\cT,\cF)$, and hence
  $\cH_T = \cT * \cF[1]$.
\end{proof}

For any silting complex $T\in\silt \UL$ the reduction $\overline T \in\silt\overline \UL$ again induces a t-structure on the derived category $\Db(\fl \overline\UL)$ of the finite dimensional algebra $\overline \UL$, with heart 
\[
  \begin{aligned}
    \cH_{\overline T} &= \{\overline E\in\Db(\fl\overline \UL) \mid \Hom_{\D(\overline\UL)}(\overline T,\overline E[i]) = 0\ \forall i\neq 0\}
    \\&= \{\overline E\in\Db(\fl\overline \UL) \mid \Hom_{\D(\UL)}(T,\overline E_\Lambda[i]) = 0\ \forall i\neq 0\}.
  \end{aligned}
\]
Via the exact embedding $(-)_\UL\colon \Db(\fl\overline \UL) \to \Db(\fl\UL)$ we can view this heart as an abelian subcategory $(\cH_{\overline T})_\UL = \cH_T \cap \Db(\fl \overline\UL)$ of $\cH_T$. Using Lemma \ref{lem:HRS} we can now show that its extension closure in $\Db(\fl\UL)$ is precisely $\cH_T$.

\begin{lemma}\label{lem:genheart}
  The heart $\cH_T$ is the extension closure of the subcategory $(\cH_{\overline T})_\UL \subset \Db(\fl\UL)$.
\end{lemma}
\begin{proof}
  By Lemma \ref{lem:HRS} the heart $\cH_T$ is generated under extension by $\cT$ and $\cF[1]$, and hence it suffices to show that $M[1],N\in \<(\cH_{\overline T})_\UL\>$ for all $N\in \cT$ and $M\in \cF$.

  Because every object $M\in \cF$ has finite length, there exists some finite filtration
  \[
    0 = \m^n M \subset \ldots \subset \m^2M \subset \m M \subset M,
  \]
  with subquotients $\m^iM/\m^{i+1}M \in\fl \overline \UL \subset \fl\UL$. 
  Recall that a torsion free class is closed under submodules, so that $\m^i M$ are again in $\cF$, and their shifts $\m^iM[1]$ therefore give elements in $\cF[1] \subset \cH_T$. 
  Because $\cH_T$ is abelian, it then also contains the quotients
  \[
    \m^iM[1]/\m^{i+1}M[1] \simeq (\m^iM/\m^{i+1}M)[1] \in \cH_T \cap (\fl \overline \UL[1]) = (\cH_{\overline T})_\UL.
  \]
  Hence, $M[1]$ is filtered by objects in $(\cH_{\overline T})_\UL$, and therefore lies in its extension closure.

  The statement for $N\in\cT$ can be proven dually, via the cofiltration 
  \[
    N = N/\m^n N \to \ldots \to N/\m^2N \to N/\m N \to 0.
  \]
  Each quotient module $N_i=N/\m^iN$ is again contained in $\cT$, because $\cT$ is a torsion class. Hence the modules $N_i$ are contained in $\cH_T$, and the subkernels $\ker(N_i \to N_{i-1})$ are therefore objects in $\cH_T \cap \fl\overline\UL \subset (\cH_{\overline T})_\UL$. It follows that $N$ is again filtered by objects of $(\cH_{\overline T})_\UL$, and therefore contained in the extension closure.
\end{proof}

Let $T\in\silt\UL$ with Krull-Schmidt decomposition $T = T_1 \oplus \ldots \oplus T_n$ and reduction $\overline T\in \silt\overline \UL$. Then, following \cite[Definition 3.8]{Asai}, the heart $\cH_{\overline T}$ is a finite length abelian category of which the simple objects are dual to the summands $\overline T_i$ of $\overline T$: for each $i=1,\ldots,n$ there is a unique simple object $\overline X_i \in \cH_{\overline T}$ for which
\[
  \Hom_{\D(\overline\UL)}(\overline T_j, \overline X_i) = 0 \quad\Longleftrightarrow\quad i\neq j.
\]
By Lemma \ref{lem:genheart}, the heart $\cH_T$ now has a similar structure: it is generated by the image of $\cH_{\overline T}$ under the exact embedding $(-)_\UL\colon \Db(\fl\overline \UL) \to \Db(\fl\UL)$, and so
\[
  \cH_T = \<(\cH_{\overline T})_\UL\> = \<X_1,\ldots,X_n\>,
\]
where $X_i = (\overline X_i)_\UL$ are the images of the simples in $\cH_{\overline T}$. In particular, $\cH_T$ is again a finite length abelian category with simples $X_i$ characterised by the condition
\[
  \Hom_{\D(\UL)}(T_j,X_i) \simeq \Hom_{\D(\overline \UL)}(\overline T_j,\overline X_i) = 0 \quad\Longleftrightarrow\quad i\neq j.
\]

Now suppose $\overline T$ is the Bongartz completion of a summand $\overline T' \subset \overline T$. Then it is shown in \cite{Asai} that for any $\overline\uptheta\in \cone^\circ\kern-1pt{\overline T'}$ the $\overline \uptheta$-stable modules in $\fl\overline\UL$ are all given by some $\overline X_i$, and in particular $\cS_{\overline \uptheta}$ is a subcategory of $\cH_{\overline T}$. Using this fact, we can now also characterise $\cS_\uptheta$ as a subcategory of the heart $\cH_T$.

\begin{prop}\label{thm:identS}
  Let $T'\in\psilt\UL$ with Bongartz completion $T\in\silt\UL$, then for all $\uptheta\in\cone^\circ\kern-1pt{T'}$ the subcategory $\cS_\uptheta$:
  \begin{enumerate}
  \item contains exactly $|T|-|T'|$ stable modules.
  \item embeds into $\Db(\fl\UL)$ as the subcategory
    \[
      \cH_T \cap (T')^\bot \colonequals \{E\in \cH_T \mid \Hom_{\D(\UL)}(T',E) = 0\}.
    \]  
  \item is the following subcategory of $\fl\UL$ 
    \[
      \{M \in \fl \UL\mid \Hom_{\D(\UL)}(T,M[1]) = \Hom_{\D(\UL)}(T',M) = 0\}.      
    \]
  \end{enumerate}
\end{prop}
\begin{proof}
  We may number the summands Krull-Schmidt decomposition $T= T_1\oplus \ldots \oplus T_n$ of $T$, such that $T' = T_{k+1} \oplus \ldots \oplus T_n$ for $k=|T|-|T'|$, and number the simples $X_i = \overline X_i$ in the heart accordingly.
  
  (1) Let $\uptheta \in \cone^\circ\kern-1pt{T'}$, then it follows from Lemma \ref{lem:conereduct} that the reduction $\overline \uptheta$ lies in $\cone^\circ\kern-1pt{\overline T'}$. As $\overline T = \overline T_1\oplus \ldots \oplus \overline T_n$ is the Bongartz completion of $\overline T_{k+1} \oplus \ldots \oplus \overline T_n$ (see Remark \ref{rem:Bongartz}) and $\overline \UL$ is finite dimensional, it follows from \cite[Proposition 4.1]{Asai} that 
  \[
    \cS_{\overline \uptheta} = \<\overline X_1,\ldots, \overline X_k\> \subset \fl\overline\UL,
  \]
  and the objects $\overline X_1, \ldots, \overline X_k$ are precisely the $\uptheta$-stable modules. By Proposition \ref{prop:stabred} the $\uptheta$-stable modules in $\cS_\uptheta \subset \fl \UL$ are exactly their images $X_i \colonequals (\overline X_i)_\UL$. There are precisely $k=|T|-|T'|$ such objects, so the result follows.
  
  (2) Because the objects $X_1,\ldots,X_n$ generate the heart $\cH_T$, it follows that $\cS_\uptheta = \<X_1,\ldots,X_k\>$ is a subcategory of $\cH_T$. Moreover, the objects $X_1,\ldots,X_k$ are precisely those simple objects in $\cH_T$ such that
  \[
    \Hom_{\D(\UL)}(T',X_i) \simeq \Hom_{\D(\overline\UL)}(\overline T_{k+1} \oplus \ldots \oplus \overline T_n, \overline X_i) = 0,
  \]
  which implies that $\cS_\uptheta$ is precisely the subcategory $\cH_T\cap (T')^\bot \subset \cH_T$.
  
  (3) As $\cS_\uptheta$ is contained in both $\fl \UL$ and $\cH_T$ it follows from Lemma \ref{lem:HRS} that $\cS_\uptheta$ is the intersection
  \[
    \begin{aligned}
      (\fl\UL) \cap \cH_T \cap (T')^\bot &= \{M\in \fl\UL \mid M\in \cT,\ \Hom_{\D(\UL)}(T',M) = 0\}\\
      &= \{M\in\fl\UL\mid \Hom_{\D(\UL)}(T,M[1])
      \\&\quad\quad\quad\quad\quad\quad\quad= \Hom_{\D(\UL)}(T',M) = 0\},
    \end{aligned}
  \]
  as claimed.
\end{proof}

Having established the subcategories $\cS_\uptheta \subset \fl\UL$ as embedded in a heart $\cH_T$ of some silting object, we can apply the derived Morita equivalence to give an explicit description of $\cS_\uptheta$ as a module category.

\begin{theorem}\label{thm:stabmod}
  Let $T'\in \psilt \UL$ with Bongartz completion $T\in\silt\UL$. Then for all $\uptheta\in \cone^\circ\kern-1pt{T'}$, the derived Morita equivalence restricts to an exact equivalence
  \[
    \cS_\uptheta \xrightarrow{\sim} \fl\End_{\Kb(\UL)}(T)/(e) \subset \fl \End_{\Kb(\UL)}(T)
  \]
  of abelian categories, where $e$ is the idempotent $e\colon T\to T'\to T$. In particular, every $\uptheta$-stable module corresponds to a simple $\End_{\Kb(\UL)}(T)$-module.
\end{theorem}
\begin{proof}
  By Lemma \ref{thm:identS}, for every $\uptheta \in \cone^\circ\kern-1pt{T'}$ the subcategory $\cS_\uptheta$ can be identified with $\cH_T \cap (T')^\bot$. By construction, the derived Morita equivalence 
  \[
    \RHom_\UL(T,-)\colon \Db(\fl \UL) \xrightarrow{\sim} \Dfl(\fE_T),
  \]
  maps the heart $\cH_T$ to the standard heart $\fl\End_{\Kb( \UL)}(T)$ and it therefore suffices to show that the subcategory $\cH_T\cap (T')^\bot \subset \cH_T$ is mapped to the subcategory
  \[
    \fl\End_{\Kb( \UL)}(T)/(e) \subset \fl \End_{\Kb( \UL)}(T)
  \]
  consisting of those $\End_{\Kb(\UL)}(T)$-modules which are annihilated by the idempotent.

  Suppose $N\in \fl\End_{\Kb(\UL)}(T)$, then $N \simeq \RHom_\UL(T,E)$ for some object $E\in\cH_T$ and therefore $N$ is in the image of $\cH_T \cap (T')^\bot$ if and only if 
  \[
    \begin{aligned}
      &\Hom_{\D(\fE_T)}(\RHom_\UL(T,T'), N) 
      \\&\quad\quad\simeq \Hom_{\D(\fE_T)}(\RHom_\UL(T,T'), \RHom_\UL(T,E))
      \\&\quad\quad\simeq \Hom_{\D(\UL)}(T',E) = 0.
    \end{aligned}
  \]
  Because $N$ is concentrated in degree $0$, while the complex $\RHom_\UL(T,T')$ has cohomology concentrated in negative degrees it follows that
  \[
    \begin{aligned}
      \Hom_{\D(\cE_T)}(\RHom_\UL(T,T'),N) 
      &\simeq \Hom_{\D(\cE_T)}(\uptau_{\geq 0}\RHom_\UL(T,T'),N) 
      \\&\simeq \Hom_{\D(\cE_T)}(\HH^0\RHom_\UL(T,T'),N) 
      \\&= \Hom_{\End_{\Kb( \UL)}(T)}(\Hom_{\Kb(\UL)}(T,T'),N)
    \end{aligned}
  \]
  where $\uptau_{\geq 0}$ denotes the canonical truncation. By inspection, $\Hom_{\Kb(\UL)}(T,T')$ is precisely the direct summand
  \[
    \Hom_{\Kb(\UL)}(T,T') \simeq e\circ \End_{\Kb(\UL)}(T) \subset \End_{\Kb(\UL)}(T),
  \]
  associated to the idempotent $e\colon T\to T'\to T$, so it follows that $N$ lies in the image of $\cH_T\cap (T')^\bot$ if and only if $N\in \fl \End_{\Kb(\UL)}(T)/(e)$.

  Hence, $\RHom_\UL(T,-)$ maps $\cS_\uptheta = \cH_T\cap (T')^\bot$ to $\fl \End_{\Kb(\UL)}(T)/(e)$, and therefore restricts to an equivalence of abelian categories
  \[
    \Hom_{\D(\UL)}(T,-)\colon \cS_\uptheta \xrightarrow{\ \sim\ } \fl\End_{\Kb( \UL)}(T)/(e).\qedhere
  \]
\end{proof}

\section{Application to cDV singularities}

We now apply the results from the previous section to the noncommutative minimal models of cDV singularities, using the relation between their tilting theory and the combinatorics of Dynkin diagrams described in Iyama--Wemyss \cite{IyWeMemoir}. The section is split up as follows.

We start by recalling the construction of the ``intersection arrangements'' in \cite{IyWeMemoir} which associates a hyperplane arrangement to a pair $(\Upgamma,J)$ of an extended Dynkin diagram and a set of vertices.

Next we recall the bijective correspondence found in \cite{IyWeMemoir} between the chambers of the intersection arrangement and the 2-term tilting complexes of a certain subalgebra of the preprojective algebra associated to $\Upgamma$.

With this background we are then able to derive several results about noncommutative minimal models of cDV singularities, both isolated and non-isolated, and give an explicit exposition using a small example.

\subsection{Intersection arrangements of extended Dynkin type}

In what follows $\Upgamma$ denotes one of the extended ADE Dynkin graphs:
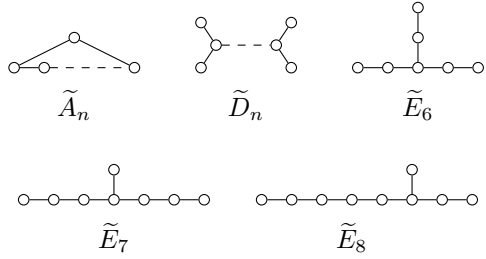
\[
  \begin{tikzpicture}[baseline=(current bounding box.center)]
    \begin{scope}[xshift=-100]
      \node at (0,-.5) {$\widetilde A_n$};
      \node[circle,draw=black,inner sep=1.4pt] (Ex) at (0,.4) {};
      \node[circle,draw=black,inner sep=1.4pt] (A) at (-.8,0) {};
      \node[circle,draw=black,inner sep=1.4pt] (B) at (-.4,0) {};
      \node[circle,draw=black,inner sep=1.4pt] (C) at (.8,0) {};
      \draw (A) to (B);
      \draw[dashed] (B) to (C); 
      \draw (C) to (Ex);
      \draw (Ex) to (A);
    \end{scope}
    \begin{scope}[xshift=-35]
      \node at (0,-.5) {$\widetilde D_n$};
      \node[circle,draw=black,inner sep=1.4pt] (A) at (-.6,.6) {};
      \node[circle,draw=black,inner sep=1.4pt] (B) at (-.4,.3) {};
      \node[circle,draw=black,inner sep=1.4pt] (C) at (-.6,0) {};
      \node[circle,draw=black,inner sep=1.4pt] (D) at (.4,.3) {};
      \node[circle,draw=black,inner sep=1.4pt] (E) at (.6,0) {};
      \node[circle,draw=black,inner sep=1.4pt] (Ex) at (.6,.6) {};
      \draw (A) to (B);
      \draw (C) to (B);
      \draw[dashed] (B) to (D);
      \draw (D) to (E);
      \draw (D) to (Ex);
    \end{scope}
    \begin{scope}[xshift=30]
      \node at (0,-.5) {$\widetilde E_6$};
      \node[circle,draw=black,inner sep=1.4pt] (A) at (-.8,0) {};
      \node[circle,draw=black,inner sep=1.4pt] (B) at (-.4,0) {};
      \node[circle,draw=black,inner sep=1.4pt] (C) at (0,0) {};
      \node[circle,draw=black,inner sep=1.4pt] (D) at (.4,0) {};
      \node[circle,draw=black,inner sep=1.4pt] (E) at (.8,0) {};
      \node[circle,draw=black,inner sep=1.4pt] (F) at (0,.4) {};
      \node[circle,draw=black,inner sep=1.4pt] (Ex) at (0,.8) {};
      \draw (A) to (B) to (C) to (D) to (E);
      \draw (Ex) to (F) to (C);
    \end{scope}
    \begin{scope}[xshift=-85,yshift=-50]
      \node at (0,-.5) {$\widetilde E_7$};
      \node[circle,draw=black,inner sep=1.4pt] (A) at (-.8,0) {};
      \node[circle,draw=black,inner sep=1.4pt] (B) at (-.4,0) {};
      \node[circle,draw=black,inner sep=1.4pt] (C) at (0,0) {};
      \node[circle,draw=black,inner sep=1.4pt] (D) at (.4,0) {};
      \node[circle,draw=black,inner sep=1.4pt] (E) at (.8,0) {};
      \node[circle,draw=black,inner sep=1.4pt] (F) at (1.2,0) {};
      \node[circle,draw=black,inner sep=1.4pt] (G) at (0,.4) {};
      \node[circle,draw=black,inner sep=1.4pt] (Ex) at (-1.2,0) {};
      \draw (Ex) to (A) to (B) to (C) to (D) to (E) to (F);
      \draw (G) to (C);
    \end{scope}
    \begin{scope}[xshift=5,yshift=-50]
      \node at (0,-.5) {$\widetilde E_8$};
      \node[circle,draw=black,inner sep=1.4pt] (A) at (-.8,0) {};
      \node[circle,draw=black,inner sep=1.4pt] (B) at (-.4,0) {};
      \node[circle,draw=black,inner sep=1.4pt] (C) at (0,0) {};
      \node[circle,draw=black,inner sep=1.4pt] (D) at (.4,0) {};
      \node[circle,draw=black,inner sep=1.4pt] (E) at (.8,0) {};
      \node[circle,draw=black,inner sep=1.4pt] (F) at (1.2,0) {};
      \node[circle,draw=black,inner sep=1.4pt] (G) at (1.6,0) {};
      \node[circle,draw=black,inner sep=1.4pt] (H) at (.8,.4) {};
      \node[circle,draw=black,inner sep=1.4pt] (Ex) at (-1.2,0) {};
      \draw (Ex) to (A) to (B) to (C) to (D) to (E) to (F) to (G);
      \draw (H) to (E);
    \end{scope}
  \end{tikzpicture}
\]

Each extended Dynkin graph has an associated hyperplane arrangement inside the vector space of real functions on $\Upgamma$, with an associated dense polyhedral fan. We recall here the construction, for which a full overview can be found in \cite{Humphreys}.

Let $\R^\Upgamma \colonequals \{f\colon \Upgamma\to \R\}$ denote the space of real functions on the vertices of $\Upgamma$. This vector space contains a standard cone $C\subset \R^\Upgamma$ of functions $f\colon \Upgamma\to \R$ which satisfy $f(v) \geq 0$ for all $v\in \Upgamma$. The cone is a union of its faces, which are the strict cones
\[
  C_V \colonequals \{f\colon \Upgamma\to \R\mid f(v) = 0 \text{ if } v\in V,\ f(v) > 0 \text{ if } v\not\in V\},
\]
where $V\subset\Upgamma$ ranges over all subsets of vertices in $\Upgamma$. Note that this includes the case $C_\emptyset$, which is the interior of $C$, and the case $C_\Upgamma = \{0\}$. The cone $C$ is a fundamental domain for the action of the affine Weyl group $W_\Upgamma$ associated to the extended Dynkin diagram, and the orbits form a subspace $\TCone(\Upgamma) = \bigcup_{w\in W_\Upgamma} w\cdot C$ called the \emph{Tits cone} of $\Upgamma$. The Tits cone is a polyhedral fan with decomposition
\[
  \TCone(\Upgamma) = \bigsqcup_{V\subset \Upgamma}\ \bigsqcup_{w\in W_\Upgamma/W_V} w\cdot C_V,
\]
with faces labelled by a subset $V$ of vertices and a coset $w$ of the stabiliser subgroup $W_V\subset W_\Upgamma$ of $C_V$. In particular, the images $w\cdot C_\varnothing$ of the interior $C_\varnothing \subset C$ are the Weyl chambers of the hyperplane arrangement. If one removes the origin from the Tits cone, the resulting space $\TCone(\Upgamma)\setminus\{0\}$ forms an open halfspace in $\R^\Upgamma$, and has a boundary given by the hyperplane $H_\infty$ dual to the imaginary root of the extended Dynkin diagram. Hence, $\R^\Upgamma$ decomposes as
\[
  \R^\Upgamma = (\TCone(\Upgamma)\setminus\{0\}) \sqcup H_\infty \sqcup (-\TCone(\Upgamma)\setminus\{0\}),
\]
where $-\TCone(\Upgamma)$ denotes the reflection of $\TCone(\Upgamma)$ in the origin. The following example for $\Upgamma = \widetilde A_1$ illustrates this decomposition:
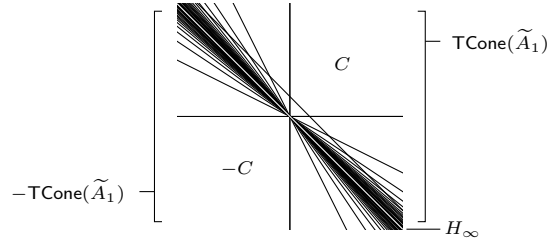
\[
  \begin{tikzpicture}[baseline=(current bounding box.center)]
    \begin{scope}
      \clip (-1.5,-1.5) rectangle (1.5,1.5);
      \draw (0,0) to (8,0);
      \node at (.7,.7) {$\scriptstyle C$};
      \node at (-.7,-.7) {$\scriptstyle -C$};
      \foreach \i in {0,1,2,3,4,5,6,7,8,9,10,12,14,16,18,20,30,40,50,60,70,80,90,100}
      {
        \draw   (4*\i,   -4*\i-4) to (-4*\i , 4*\i+4);
        \draw   (4*\i,   -4*\i+4) to (-4*\i , 4*\i-4);
      }
      \draw[white] (-1.5,1.5) to (1.5,-1.5);
    \end{scope}
    \draw[thin] (1.7,1.4) to (1.8,1.4) to (1.8,-1.4) to (1.7,-1.4);
    \draw[thin] (1.8,1) to (2,1);
    \node at (2.8,1) {$\scriptstyle \TCone(\widetilde A_1)$};
    \draw[thin] (-1.7,1.4) to (-1.8,1.4) to (-1.8,-1.4) to (-1.7,-1.4);
    \draw[thin] (-1.8,-1) to (-2,-1);
    \node at (-2.95,-1) {$\scriptstyle -\TCone(\widetilde A_1)$};
    \draw[thin] (1.55,-1.5) to (2,-1.5);
    \node at (2.3,-1.5) {$\scriptstyle H_\infty$};
  \end{tikzpicture}
\]
Given a fixed subset $J\subset \Upgamma$ of vertices, we consider the linear subspace $\R^J \subset \R^\Upgamma$ of dimension $|J|$ that consists of all functions $f\colon \Upgamma\to \R$ which vanish on the complement $\Upgamma\setminus J$. Following \cite{IyWeMemoir} we define the Tits cone 
\[
  \TCone(\Upgamma,J) \colonequals \bigsqcup_{V\subset \Upgamma} \bigsqcup_{\substack{w\in W_\Upgamma/W_V\\w\cdot C_V \subset \R^J}} w\cdot C_V,
\]
which is the subfan of $\TCone(\Upgamma)$ consisting of all faces that lie inside the subspace $\R^J$. The vector space $\R^J$ now has an analogous decomposition
\[
  \R^J = \left(\TCone(\Upgamma,J)\setminus \{0\}\right) \sqcup H_\infty^J \sqcup \left(-\TCone(\Upgamma,J)\setminus \{0\}\right)
\]
into the positive and negative Tits cones, separated by the hyperplane $H_\infty^J = H_\infty \cap \R^J$. The \emph{intersection arrangement} of the pair $(\Upgamma,J)$ is defined as the union of the positive and negative cone:
\[
  X_{\Upgamma,J} = \TCone(\Upgamma,J) \cup -\TCone(\Upgamma,J).
\]
In what follows we also write $X_k$ for the codimension $k$ stratum of a given $X = X_{\Upgamma,J}$, which consists of the faces $w\cdot C_V$ and $w\cdot (-C_V)$ for which $V$ satisfies $|V| = k + |J|$.

\begin{example}
  Let $\Upgamma = \widetilde D_4$ and choose $J \subset \Upgamma$ to consist of the extended and middle vertex (indicated by the black nodes below), then $X_{\Upgamma,J}$ has the following structure:
  
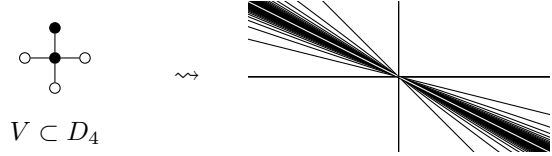
\[
    \begin{tikzpicture}[baseline=(current bounding box.center)]
      \begin{scope}[xshift=-50]
        \node at (0,-1+.25) {$V\subset D_4$};
        \node[circle,fill=black,draw=black,inner sep=1.4pt] (A) at (0,0+.25) {};
        \node[circle,fill=black,draw=black,inner sep=1.4pt] (B) at (0,.4+.25) {};
        \node[circle,draw=black,inner sep=1.4pt] (C) at (.4,0+.25) {};
        \node[circle,draw=black,inner sep=1.4pt] (D) at (-.4,0+.25) {};
        \node[circle,draw=black,inner sep=1.4pt] (E) at (0,-.4+.25) {};
        \draw (B) to (A) to (E);
        \draw (C) to (A) to (D);
      \end{scope}
      \node at (0,0) {$\leadsto$};
      \begin{scope}[xshift=80]
        \clip (-2,-1) rectangle (2,1);
        \draw (0,0) to (8,0);
        \foreach \i in {0,1,2,3,4,5}
        {
          \draw   (16*\i,   -8*\i-4) to (-16*\i , 8*\i+4);
          \draw   (16*\i,   -8*\i+4) to (-16*\i , 8*\i-4);
          \draw   (12*\i-6, -6*\i) to (-12*\i+6 ,6*\i);
          \draw   (12*\i+6, -6*\i) to (-12*\i-6, 6*\i);
        }
        \foreach \i in {6,7,8,9,10,11,12,13,14,15,16,17,18,19,20,21,22,23,24,25,26,27,28,29,30,35,40,45,50,55,60,70,80,90}
        {
          \draw   (4*\i,   -2*\i-1) to (-4*\i , 2*\i+1);
          \draw   (4*\i,   -2*\i+1) to (-4*\i , 2*\i-1);
          \draw   (2*\i-1, -\i) to (-2*\i+1, \i);
          \draw   (2*\i+1, -\i) to (-2*\i-1, \i);
        }
        \draw[color=white] (12, -6) to (-12,6);
        \draw (-2,0) to (2,0);
        \draw (0,-1) to (0,1);
      \end{scope}
    \end{tikzpicture}
  \]
  The stratum $X_0$ consists of the Weyl chambers, while $X_1$ is the union of the strict rays emanating from the origin, and $X_2$ is the origin. The hyperplane $H_\infty^J$ is the line dual to the vector $(1,2)$, which is a restriction of the imaginary root of $\widetilde D_4$.
\end{example}

\subsection{Contracted preprojective algebras}

For a fixed extended Dynkin diagram $\Upgamma$, let $Q = (Q_0,Q_1)$ be the quiver which has a vertex $v\in Q_0$ for every vertex in $\Upgamma$, and a pair of arrows $a\colon v \to w$, $a^*\colon w\to v$ in $Q_1$ for every edge between $v$ and $w$ in $\Upgamma$. For a given field $\KK$, consider the preprojective algebra
\[
  \KK Q/(\textstyle\sum_a aa^* - a^*a).
\]
It is well known that the preprojective algebra is finite as a module over its centre, which is the coordinate ring of an ADE surface singularity of the corresponding Dynkin type. We let $\Uppi = \Uppi_\Upgamma$ denote the completion of the preprojective algebra at this singularity, which is then a module finite algebra over the complete local ring $Z(\Uppi)$. 

There is an idempotent $e_v\in \Uppi$ for every vertex $v\in \Upgamma$ and together the $e_v$ form a complete set of primitive orthogonal idempotents for $\Uppi$. Hence, for each pair $(\Upgamma,J)$ as in the previous section there is an idempotent $e = e_J = \sum_{v\in J} e_v$ and, following \cite{IyWeMemoir}, we define the \emph{contracted preprojective algebra}
\[
  e\Uppi e = e_J\Uppi e_J \subset \Uppi,
\]
generated by the paths starting and ending in a vertex of $J$. In \cite{IyWeMemoir} it is shown that the hyperplane arrangement $X_{\Upgamma,J}$ of a pair $(\Upgamma,J)$ is related to the silting fan of these algebras via the natural identification
\begin{equation}\label{eq:Kthident}
  \textstyle\K_0(\proj e \Uppi e)_\R \simeq \bigoplus_{v\in J} \R [e_v\Uppi e] \simeq \R^J,
\end{equation}
of the classes of indecomposable projectives $e_v\Uppi e$ with the corresponding basis vectors in $\R^J$. We recall this theorem here in our current notation.
\begin{theorem}[{\cite[Theorem 7.24]{IyWeMemoir}}]\label{thm:siltfan}
  For each preprojective algebra $e\Uppi e$ associated to a pair $(\Upgamma,J)$, the isomorphism \eqref{eq:Kthident} identifies the fan $X_{\Upgamma,J}$ with the silting fan $\SFan{e\Uppi e}$. In particular, for every $T\in\psilt e\Uppi e$ the strict cone is given by
  \[
    \cone^\circ{T} = w \cdot C_V,
  \]
  for some subset $V\subset \Upgamma$ and $w\in W_\Upgamma/W_V$, such that $w\cdot C_V$ lies in $\R^J$.
\end{theorem}

In particular, Proposition \ref{thm:identS} now yields the number of stable $e \Uppi e$-modules for stability conditions in the fan.

\begin{prop}\label{prop:nmodules}
  Let $(\Upgamma,J)$ be a Dynkin type with fan $X = X_{\Upgamma,J}$ as above. Then for $\uptheta \in X_k \subset \R^J$ the subcategory $\cS_\uptheta \subset \fl e\Uppi e$ is generated by $k$ stable modules, and only depends on the face $w\cdot C_V$ in which $\uptheta$ lies.
\end{prop}

The above proposition includes the case where $e\Uppi e = \Uppi$ is the entire preprojective algebra, for which this result was already established by Sekiya--Yamaura \cite{SeYaPreProj} via a similar tilting method.

\subsection{Compound Du Val singularities}

In dimension three, the natural analogue of Du Val singularities are the \emph{compound Du Val (cDV) singularities}, which have a coordinate ring that is (up to isomorphism) of the form
\[
  R = \C[\![x,y,z,t]\!]/(f(x,y,z) + t \cdot g(x,y,z,t)),
\]
where $f$ defines the Du Val surface singularity $\Spec \C[\![x,y,z]\!]/(f) = \Spec R/(t)$. These types of singularities are a basic building block of the minimal model program, as they form the base of various curve contractions.

In \cite{IyWeAusRei} Iyama--Wemyss characterise a type of \emph{noncommutative} minimal model for cDV singularities. Concretely, they define a set of $R$-modules $\MM R$ called the \emph{maximal modifying modules}, which are basic reflexive $R$-modules whose endomorphism algebras $\Lambda_M \colonequals \End_R(M)$ have similar homological properties to the geometric minimal models.
In particular, $\Lambda_M$ is a symmetric $R$-order, giving it a ``singular 3-Calabi--Yau'' property \cite[Theorem 3.2]{IR}.
 Moreover, these endomorphism algebras are deformations of a contracted preprojective algebra: by \cite[Proposition 9.4(1)]{IyWeMemoir} there are isomorphisms
\begin{equation}\label{eq:Dtype}
  \Lambda_M / t \Lambda_M \simeq e\Uppi e = e_J \Uppi_\Upgamma e_J,
\end{equation}
for each $M\in\MM R$, where $\Upgamma$ is the extended Dynkin diagram of the Du Val singularity $\Spec R/(t)$ and $J\subset \Upgamma$ a subset which may depend on $M$. 

By applying Proposition \ref{prop:fanquot} to the quotient $\Lambda_M \to e\Uppi e$, we recover the silting fan of any minimal model $\Lambda_M$.
The singular Calabi--Yau property of $\Lambda_M$ moreover implies that all silting complexes are \emph{tilting}, so that we find the following refinement of \cite[Theorem 9.6]{IyWeMemoir}.
\begin{prop}\label{prop:cDVfan}
  Let $M\in\MM R$ with associated Dynkin type $(\Upgamma,J)$. Then there is an isomorphism $\K_0(\proj \Lambda_M)_\R \xrightarrow{\sim} \R^J$, which identifies $\SFan{\Lambda_M}$ with $X_{\Upgamma,J}$. In particular, the chambers of $X_{\Upgamma,J}$ are in bijection with basic 2-term tilting complexes for $\Lambda_M$.
\end{prop}
\begin{proof}
  The endomorphism algebra $\Lambda_M = \End_R(M)$ is finite as a module over the complete local ring $R$, so it follows from Proposition \ref{prop:fanquot} that the isomorphism
  \[
    \K_0(\proj \Lambda_M)_\R \xrightarrow{\ \sim\ } \K_0(\proj \Lambda_M/t\Lambda_M)_\R \simeq \K_0(\proj e\Uppi e)_\R
  \]
  induced by the quotient $\Lambda_M \to \Lambda_M/t\Lambda_M \simeq e\Uppi e$ onto the contracted preprojective algebra associated to $(\Upgamma,J)$ identifies the fans $\SFan{\Lambda_M}$ and $\SFan{e\Uppi e}$. The result now follows directly from Theorem \ref{thm:siltfan}, as the isomorphism $\K_0(\proj e\Uppi e)_\R \xrightarrow{\sim} \R^J$ identifies $\SFan{e\Uppi e}$ with $X_{\Upgamma,J}$.

  In particular, each chamber of $X_{\Upgamma,J}$ corresponds to a basic 2-term silting complex in $\silt\Lambda_M$. But it follows from \cite[Proposition A.2]{KimMiz} that any silting complex for a symmetric $R$-order is tilting, which yields the claimed bijection.
\end{proof}

In what follows we identify $\K_0(\proj \Lambda_M)_\R$ with $\R^J$ and $\SFan{\Lambda_M}$ with the fan $X_{\Upgamma,J}$ of the intersection arrangement. Theorem \ref{thm:identS} now yields the following. 

\begin{prop}\label{prop:cDVnumstables}
  Let $M\in \MM R$ with fan $X = X_{\Upgamma,J}$. Then for $\uptheta \in X_k$ the subcategory $\cS_\uptheta\subset \fl\Lambda_M$ is generated by $k$ stable modules and depends only on the face $w\cdot C_V$ in which $\uptheta$ lies. In particular:
  \begin{itemize}
  \item there are no $\uptheta$-stable modules if $\uptheta$ lies in a Weyl chamber,
  \item there is a unique $\uptheta$-stable module if $\uptheta$ lies generically on a hyperplane.
  \end{itemize}
\end{prop}

\begin{rem}
Proposition \ref{prop:cDVnumstables} describes the stability everywhere in $\R^J$ except on the boundary hyperplane $H_\infty^J$.
By \cite[Proposition 9.4]{IyWeMemoir}, this hyperplane is dual to the ``rank vector'' $(\operatorname{rk} M_i)_{i\in J}$ where $M_i$ is the indecomposable summand of $M$ corresponding to a node $i\in J$.
The work of Wemyss \cite{HomMMP} shows that the stability parameters on this hyperplane fall into GIT chambers for the (geometric) minimal models of the singularity $\Spec R$, and form a finite version of the affine intersection arrangement used here.
The significance of the proposition is thus that it now also allows us to describe the stability for \emph{all other} parameters.
\end{rem}

\subsection{The isolated case}

Now assume that $R$ is an isolated cDV singularity, and fix a minimal model $\Lambda = \Lambda_M$ for some $M\in \MM R$ with Dynkin type $(\Upgamma,J)$.

For isolated cDV singularities Iyama-Wemyss \cite[\S9.2]{IyWeMemoir} give a complete classification of the 2-term tilting complexes via an Auslander-McKay type correspondence: by \cite[Theorem 4.17]{IyWeAusRei} there is a bijection
\begin{equation}\label{eq:AusMcK}
  \{ N\in\MM R \} \quad\underset\sim{\xrightarrow{\ \Hom_R(M,-) \ }}\quad \{\text{basic tilting $\Lambda$-modules}\},
\end{equation}
which assigns a (classical) tilting module $\Hom_R(M,N)\in \mod \Lambda$ to each $N \in \MM R$ with endomorphism algebra
\[
  \End_\Lambda(\Hom_R(M,N)) \simeq \End_R(N) = \Lambda_N.
\]
Because each $\Hom_R(M,N)$ is a classical tilting module, it has projective dimension $\leq 1$ and the projective resolution
\[
  \ldots \to 0 \to P^{-1}_N \to P^0_N \to \Hom_R(M,N)
\]
therefore defines a 2-term tilting complex $P_N \in \silt \Lambda$.
According to \cite[Theorem 9.8]{IyWeMemoir} the cones of these complexes fill out the Tits cone. The negative cone can be obtained by dualising: the dual statement of \eqref{eq:AusMcK} yields a tilting module $\Hom_R(N,M) \in \mod \Lambda^\op$ for the opposite algebra and the resolution
\[
  \ldots \to 0 \to Q^{-1}_N \to Q^0_N \to \Hom_R(N,M)
\]
yields an object $Q_N \in \silt \Lambda^\op$. By applying the $R$-linear dual $(-)^* = \Hom_R(-,R)$ and shifting, one then obtains a new 2-term tilting complex for $\Lambda$.

\begin{lemma}\label{lem:allsilts}
  For every $T\in \silt \Lambda$ there exists some $N\in \MM R$ such that $T$ is either isomorphic to $P_N$ \'or isomorphic to $Q_N^*[1]$.
\end{lemma}
\begin{proof}
  Let $(\Upgamma,J)$ denote the Dynkin type of $\Lambda$, then \cite[Corollary 9.8]{IyWeMemoir} shows that the cones of $P_N$ form the Tits cone:
  \[
    \bigcup_{N\in\MM R} \cone{P_N} = \TCone(\Upgamma,J).
  \]
  Hence if $T\in \silt \Lambda$ with $\cone{T} \subset \TCone(\Upgamma,J)$, then Proposition \ref{prop:chamberwalls} shows that $T$ must be isomorphic to $P_N$ for some $N\in \MM R$.

  Consider now the dual tilting complexes $Q_N \in \silt \Lambda^\op$. The $R$-linear dual induces an exact anti-equivalence $(-)^*\colon \Kb( \Lambda^\op)^\op \to \Kb( \Lambda)$. Hence for every the $N\in \MM R$ there is a tilting complex
  \[
    Q_N^* =\ \ldots \to 0 \to \Hom_R(Q_N^0,R) \to \Hom_R(Q_N^{-1},R) \to 0 \to \ldots    
  \]
  which is concentrated in degrees $0,1$. Hence the objects $Q_N^*[1]$ are indeed in $\silt\Lambda$. It is shown in \cite[Theorem 9.17]{IyWeMemoir} that the K-theory classes of the summands of $Q_N^*$ are equal to those of the summands of $P_N$, which implies that $\cone(Q_N^*[1]) = -\cone{P_N}$. Hence, the union of these cones fill out the negative Tits cone: 
\[
  \bigcup_{N\in\MM R} \cone(Q_N^*[1]) = \bigcup_{N\in\MM R} -\cone{P_N} = -\TCone(\Upgamma,J).
\]
Hence, if $\cone{T}$ is contained in $-\TCone(\Upgamma,J)$, then it again follows that $T\simeq Q_N^*[1]$ for some $N\in \MM R$. By Proposition \ref{prop:cDVfan} the silting fan decomposes as $\SFan{\Lambda} = X_{\Upgamma,J} = \TCone(\Upgamma,J) \cup -\TCone(\Upgamma,J)$, so there are no other cases to consider.
\end{proof}

Using Theorem \ref{thm:stabmod} we can now derive an explicit form for the subcategories $\cS_\uptheta$ for all $\uptheta$ in the intersection arrangement, as module categories for quotients of noncommutative minimal models.

\begin{theorem}\label{thm:cDVstabs}
  Suppose $\Lambda = \Lambda_M$ is a noncommutative minimal model of an isolated cDV singularity $R$ as above. Then for each $\uptheta \in X_{\Upgamma,J}$ there exists an $N\in \MM R$ and an idempotent $e\in \Lambda_N$ such that
  \[
    \cS_\uptheta \simeq \fl \Lambda_N/(e),
  \]
  as abelian categories.
\end{theorem}
\begin{proof}
  For $\uptheta \in X_{\Upgamma,J} = \SFan{\Lambda}$, Theorem \ref{thm:stabmod} implies that
  \[
    \cS_\uptheta \simeq \fl \End_{\Kb( \Lambda)}(T)/(e)
  \]
  for some $T\in\silt\Lambda$ and the idempotent $e\colon T\to T'\to T$ associated to a summand $T'\subset T$. By Lemma \ref{lem:allsilts} any $T\in \silt \Lambda$ is isomorphic to either some $P_N$, which has an endomorphism algebra
  \[
    \End_{\Kb( \Lambda)}(P_N) \simeq \End_\Lambda(\Hom_R(M,N)) \simeq \Lambda_N,
  \]
  or $T$ is isomorphic to $Q_N^*[1]$, which has endomorphism algebra
  \[
    \begin{aligned}
      \End_{\Kb( \Lambda)}(Q_N^*[1]) 
      &\simeq \End_{\Kb( \Lambda^\op)}(Q_N)^\op
      \\&\simeq \End_{\Lambda^\op}(\Hom_R(N,M))^\op \simeq \Lambda_N.
    \end{aligned}
  \]
  Hence $\End_{\Kb(\Lambda)}(T) \simeq \Lambda_N$ for some $N\in \MM R$ as claimed.
\end{proof}

Finally we consider the special case where $R$ is the base of a threefold flopping contraction $f\colon Y\to \Spec R$ with $Y$ smooth. In this setting, Hirano--Wemyss \cite[\S7.2]{HiWeStab} show that the set $\MM R$ is freely acted on by a group $\Z^n = \<L_1,\ldots,L_n\> \subset \Cl(R)$ of invertible ideals, preserving the endomorphism algebras: for each $N\in\MM R$ there are isomorphisms
\[
  \Lambda_{L_i\cdot N} = \End_R(L_i \cdot N) \simeq \End_R(N) = \Lambda_N
\]
for each generator $L_i$, where $N \mapsto L_i\cdot N \in \MM R$ denotes the action. They moreover show that the action partitions $\MM R$ into finitely many orbits, so that there are only finitely many noncommutative minimal models $\Lambda_N$ up to isomorphism. We can therefore give the following strengthening of Theorem \ref{thm:cDVstabs} for flops.

\begin{prop}\label{prop:cDVstabs}
  Suppose $R$ is the base of a threefold flopping contraction. Then there is a finite set $\{N^1,\ldots,N^k\} \subset \MM R$ such that: for each $\uptheta \in X_{\Upgamma,J}$ there is an equivalence $\cS_\uptheta \simeq \fl \Lambda_{N^i}/(e)$ for some $i=1,\ldots,k$ and idempotent $e \in \Lambda_{N^i}$.
\end{prop}

We finish with the example of an $A_2$ flop, where we can explicitly determine the set $\{N^1,\ldots,N^k\}$ and the quotients of the endomorphism algebras $\Lambda_{N^i}$.

\begin{example} We recall the example \cite[Example 6.1]{DoWeContDef} where $R$ is the base of an $A_2$ flopping contraction, in which two rational curves in a smooth threefold are contracted to the cDV singularity
  \[
    R = \C[\![x,y,z,t]\!]/(xy(x+y) - t z)
  \]
  As in \cite{DoWeContDef} we may pick the MM module $M= R\oplus (u,x) \oplus (u,xy)$, which has an endomorphism algebra isomorphic to the Jacobi algebra $\C(Q,W)$ of the quiver with potential
  \[
    \begin{tikzpicture}
      \node at (-2.5,.9) {$Q:$};
      \node[circle,draw=black,inner sep=1.4pt,outer sep=4pt] (A) at (-1.5*.85,0) {};
      \node[circle,draw=black,inner sep=1.4pt,outer sep=4pt] (B) at (0,2.25*.85) {};
      \node[circle,draw=black,inner sep=1.4pt,outer sep=4pt] (C) at (1.5*.85,0) {};
      \draw[->] (A) to[bend left=15,edge label=$\scriptstyle a_1$] (B);
      \draw[->] (B) to[bend left=15,edge label=$\scriptstyle b_1$,pos=.45] (A);
      \draw[->] (C) to[bend left=15,edge label=$\scriptstyle b_2$,pos=.55] (B);
      \draw[->] (B) to[bend left=15,edge label=$\scriptstyle a_2$] (C);
      \draw[->] (A) to[bend left=15,edge label=$\scriptstyle b_3$] (C);
      \draw[->] (C) to[bend left=15,edge label=$\scriptstyle a_3$] (A);
      \node at (5,.9) {$W = \textstyle\tfrac12 \sum_{i=1}^3  a_ib_ia_ib_i + a_2a_1b_1b_2$};
      \node at (5.1,.3) {$- a_1a_3b_3b_1 - a_3a_2b_2b_3.$};
    \end{tikzpicture}
  \]
  The Dynkin type $(\Upgamma,J)$ of $M$ is the diagram $\Upgamma = A_2$ with $J$ the full set of its vertices. In particular, the intersection arrangement is just the ordinary $\widetilde A_2$ hyperplane arrangement $X_{\Upgamma,J}\subset \R^3$, and the Tits cone $\TCone(\Upgamma,J)$ fills the halfspace $\{r\in \R^3\mid r_1 +r_2+r_3 > 0\}$. As in \cite{HiWeStab} we can visualise the Tits cone by intersecting it with a level $\{r \in \R^3 \mid r_1 + r_2 + r_3 = 1\}$, which yields the affine $A_2$ lattice
  
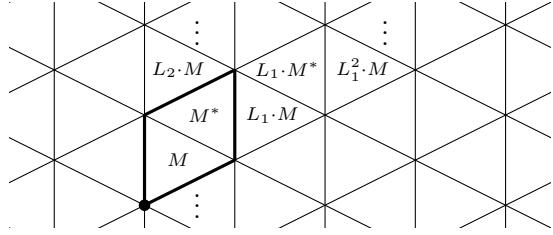
\[
    \begin{tikzpicture}
      \begin{scope}[xscale=1.2,yscale=.6]
      \clip (-2.5,-1.5) rectangle (3.5,3.5);
      \foreach \i in {-3,...,3}
        \draw (\i,-4) to (\i,4);
      \foreach \i in {-3,...,3} {
        \draw (2*\i-4,-4) to (2*\i+4,4);
        \draw (2*\i+4,-4) to (2*\i-4,4);
      }
      \node[circle,fill=black,draw=black,inner sep=1.4pt] at (-1,-1) {};
      \draw[very thick] (-1,-1) to (0,0);
      \draw[very thick] (-1,-1) to (-1,1);
      \draw[very thick] (-1,1) to (0,2);
      \draw[very thick] (0,0) to (0,2);
      \end{scope}
      \node at (-.75,0) {$\scriptstyle M$};
      \node at (.5,.6) {$\scriptstyle L_1\cdot M$};
      \node at (-.4,.6) {$\scriptstyle M^*$};
      \node at (.7,1.2) {$\scriptstyle L_1\cdot M^*$};
      \node at (-.75,1.2) {$\scriptstyle L_2\cdot M$};
      \node at (1.7,1.2) {$\scriptstyle L_1^2\cdot M$};
      \node at (2,1.8) {$\vdots$};
      \node at (-.5,1.8) {$\vdots$};
      \node at (-.5,-.5) {$\vdots$};
    \end{tikzpicture}
  \]
  Here the thick border indicates a fundamental region for the action $\Z^2 \simeq \<L_1,L_2\>$ on the level, which contains precisely two chambers: one for $M$ itself and one for its dual $M^*$. There are thus two endomorphism algebras to consider:
  \[
    \End_R(M) \simeq \C(Q,W) \quad\text{and}\quad \End_R(M^*) \simeq \C(Q,W)^\op.
  \]
  From the presentation of the quiver above, one can however observe that $\C(Q,W)^\op$ and $\C(Q,W)$ are isomorphic.

  If $\uptheta \in X_{\Upgamma,J}$ is zero then $\cS_\uptheta = \fl \C(Q,W)$ because every module is semistable, while $\cS_\uptheta = 0$ if $\uptheta$ lies in a Weyl chambers.
  In all other cases $\uptheta$ lies on a wall, and Proposition \ref{prop:cDVstabs} shows that $\cS_\uptheta \subset \fl \Lambda$ is equivalent to $\fl \End_R(M)/(e)$ for $e$ the idempotent of a proper nonzero summand of $M$. The isomorphism $\C(Q,W) \simeq \End_R(M)$ identifies this quotient with the Jacobi algebra $\C(Q',W)$, where $Q'$ is obtained from $Q$ by deleting the vertices corresponding to the summand of $M$. Therefore, $\cS_\uptheta$ is the module category of one of the following Jacobi algebras:
  \[
    \C(\begin{tikzpicture}
      \clip (-.2,.2) rectangle (.2,-.1);
      \node[circle,draw=black,inner sep=1.4pt,outer sep=4pt] at (0,0) {};
    \end{tikzpicture},0) \simeq \C,
    \quad
    \C\Big(\begin{tikzpicture}[baseline=(current bounding box.center)]
      \clip (-.7,.6) rectangle (.7,-.8);
      \node[circle,draw=black,inner sep=1.4pt,outer sep=4pt] (A) at (-.5,0) {};
      \node[circle,draw=black,inner sep=1.4pt,outer sep=4pt] (B) at (.5,0) {};
      \draw[->] (A) to[bend left,edge label=$\scriptstyle a$] (B);
      \draw[->] (B) to[bend left,edge label=$\scriptstyle b$] (A);
    \end{tikzpicture},\tfrac12 abab\Big).
  \]
\end{example}


\end{document}